\documentclass[10pt,a4paper]{amsart}

\usepackage{amsthm, amsfonts, amssymb, amsmath, latexsym, enumerate,array}
\usepackage{ulem}
\usepackage{amscd} 
\usepackage[all]{xy}
\usepackage[utf8]{inputenc}
\usepackage{hyperref,cite,mdwlist,mathrsfs}
\usepackage{tikz}
\usepackage{pgf,tikz}
\usetikzlibrary{arrows}

\usepackage{anysize}
\usepackage{setspace}

\newtheorem{THM}{Theorem}
\newtheorem{thm}{Theorem}[section]
\newtheorem{lem}[thm]{Lemma}
\newtheorem{corol}[thm]{Corollary}

\newtheorem*{thm*}{Theorem}
\newtheorem*{cnj*}{Conjecture}

\theoremstyle{definition}
\newtheorem{rmk}[thm]{Remark}
\newtheorem{eg}[thm]{Example}

\newtheorem*{conj*}{Conjecture}

\newcommand{\sO}{\mathscr{O}}

\newcommand{\sF}{\mathscr{F}}

\newcommand{\cF}{\mathcal{F}}

\newcommand{\cI}{\mathcal{I}}

\DeclareMathOperator{\HH}{H}
\DeclareMathOperator{\hh}{h}

\newcommand{\C}{\mathbb C}

\newcommand{\N}{\mathbb N}
\newcommand{\p}{\mathbb P}
\newcommand{\PD}{\check{\mathbb P}}

\newcommand{\mult}{\mathrm{mult}}

\usetikzlibrary{arrows}

\allowdisplaybreaks[1]
\begin{document}



\title{On invariant rank two vector bundles on $\mathbb{P}^2$}

\author{Simone Marchesi} 
\email{{\tt marchesi@ub.edu}}
\address{Universitat de Barcelona\\
  Gran Via de les Corts Catalanes, 585 - 08007 Barcelona - Spain}

\author{Jean Vallès}
\email{{\tt jean.valles@univ-pau.fr}}
\address{Université de Pau et des Pays de l'Adour \\
  Avenue de l'Université - BP 576 - 64012 PAU Cedex - France}
\urladdr{\url{http://jvalles.perso.univ-pau.fr/}}


\thanks{Both authors were partially supported by CAPES-COFECUB project 926/19.}

\begin{abstract} 
In this paper we characterize the rank two vector bundles  on $\p^2$ which are invariant under the actions of  the parabolic subgroups
 $G_p:=\mathrm{Stab}_p(\mathrm{PGL}(3))$ fixing a point in the projective plane,  $G_L:=\mathrm{Stab}_L(\mathrm{PGL}(3))$ fixing a line, 
and when $p\in L$,  the Borel subgroup $\mathbf{B} = G_p \cap G_L$ of $\mathrm{PGL}(3)$.
Moreover, we prove that the geometrical configuration of the jumping locus induced by the invariance does not, on the other hand, 
characterize the invariance itself. Indeed, we find infinite families that are \textit{almost uniform} but not \textit{almost homogeneous}.
\end{abstract}

\maketitle

\section{Introduction}
The description and classification of vector bundles, which are invariant under the action of a determined group,
has been widely studied. For instance, rank $r$ vector bundles over $\p^n$ which are invariant under the canonical action of $\mathrm{PGL}(n+1,\mathbb{C})$ 
are called \textit{homogeneous}. Their complete classification is known up to rank $n+2$, see \cite{ElliaUnif} for a reference, and they are given only by direct 
sums involving line bundles, a twist of the tangent bundle, a twist of the cotangent bundle on $\p^n$ or their symmetric or anti-symmetric powers.

Furthermore, particular situations induce us to consider the action of specific subgroups of the projective linear group. For example, 
Ancona and Ottaviani prove in \cite{AO-Unstable} that the Steiner bundles on $\p^n$ which are invariant under the action of the special linear group $\mathrm{SL}(2,\mathbb{C})$
are the ones introduced by Schwarzenberger in \cite{Schw-vbundles}. Further in this direction, the second named author proves
in \cite{Valles} that any rank 2 stable vector bundle on $\p^2$ which is invariant under the action of $\mathrm{SL}(2,\mathbb{C})$ is a Schwarzenberger bundle.

In this paper we will consider rank 2 vector bundles on the projective plane $\p^2$, and the chosen subgroups of $\mathrm{PGL}(3)=\mathrm{PGL}(3,\mathbb{C})$
have been inspired by the following observations.\\
In a previous paper, see \cite{MV}, both authors have studied \textit{nearly free vector bundles} coming from line arrangements.\\ 
First of all, recall that it is of great interest also the studying and description of the action of a group on a hyperplane arrangement. For example, hyperplane arrangements
which are invariant under the action of the group defined by reflections are free, and therefore their associated vector bundle is a direct sum of line bundles and hence homogeneous 
(see \cite{OrlikTerao-Arrangements} for more details).\\
 Recall moreover that nearly free vector bundles $\mathcal{F}$, which were introduced by Dimca and Sticlaru in \cite{DimcaSticlary-NF}, can be defined by the short exact sequence
 \begin{equation}
\label{nf}
0\longrightarrow \sO_{\p^2}(-b-1) \stackrel{M}{\longrightarrow} \sO_{\p^2}(-a) \oplus \sO_{\p^2}(-b) ^2 \longrightarrow \mathcal{F} \longrightarrow 0,
 \end{equation}
with $(a,b)\in \mathbb{N}^2$ called the exponents of the vector bundle.\\
In particular, we proved the following two results
\begin{itemize}
\item Let $\mathcal{F}$ be a nearly free vector bundle. Then, there exists a point $p$ such that a line $l\in \p^2$ is a jumping line of $\mathcal{F}$
if and only if $l$ passes through $p$. Moreover, each jumping line has order of jump equal to $1$. We call $p$ the \textit{jumping point} associated to the vector bundle.
\item Given a point $p \in \p^2$ and a couple of integers $(a,b)\in \mathbb{C}^2$ with $ a\le b$, there exists, up to isomorphism, one and only one nearly free vector bundle
with exponents $(a,b)$ whose pencil of jumping lines has $p$ as base point. Moreover, we can think of its defining matrix $M$ as
$$
{}^tM = [x,y,z^{b-a+1}].
$$
\end{itemize}
Furthermore, we proved that the geometrical configuration of the jumping locus $S(\mathcal{F})$, described in the first item, 
``almost'' characterises nearly free vector bundles (see \cite[Theorem 2.8]{MV}).\\ 

Inspired by the essential nature of the jumping point $p$, we focus on the rank 2 vector bundles on $\p^2$ which are invariant under the action of the subgroup $G_p \subset \mathrm{PGL}(3)$ that 
fixes the point $p$ in the projective plane. 

Assume that $p=(1:0:0)$.  The isotropy groups $G_p$ and $G_L$, that fixes the line $L=\{z=0\}$, are maximal parabolic subgroups containing the Borel subgroup $\mathbf{B} = G_p \cap G_L$ of upper triangular matrices fixing $p$ and $L$, with $p\in L$. So the question of invariance under the action of $G_p$ naturally extends to $G_L$ and $\mathbf{B}$.

The Borel subgroup $\bf{B}$ of upper triangular matrices contains the maximal torus $\bf{T}$ of diagonal matrices. As it was kindly pointed out by the referee, the classification of indecomposable rank two $\bf{T}$-invariant bundles has been done by Kaneyama in \cite{KaneyamaP2} (see also \cite{KaneyamaPn} for a generalization in higher dimension). Specifically, he proves that a $\bf{T}$-invariant vector bundle is isomorphic to a twist of the vector bundle $E(a,b,c)$ (where $a,b,c$ are positive integers) defined by 
\begin{equation}\label{bundle-torusinvariant}
0\longrightarrow   \sO_{\p^2}\stackrel{(x^a,y^b,z^c)}\longrightarrow  \sO_{\p^2}(a) \oplus \sO_{\p^2}(b)\oplus \sO_{\p^2}(c)\longrightarrow  E(a,b,c) \longrightarrow  0.
\end{equation}

Since $\mathbf{T}\subset \mathbf{B}=G_p\cap G_L$ the indecomposable bundles invariant under the action of either $\bf{B}$, $G_p$ or $G_L$ belong to the family studied by Kaneyama. Furthermore, a nearly free vector bundle defined by the exact sequence (\ref{nf}) is isomorphic to $E(1,1,b-a+1)\otimes \sO_{\p^2}(-b-1)$.

We deal first with the $\mathbf{B}$-invariant case, deducing the others from this one. The results obtained can be concentrated in the following statement. 
\begin{THM}
Let $\cF$ be an indecomposable rank two vector bundle on $\p^2$. Then
\begin{itemize}
\item $\cF$ is invariant under the action of $G_p$ if and only if it is a nearly free vector bundle with jumping point $p$;
\item $\cF$ is invariant under the action of $G_L$ if and only if it is homogeneous;
\item $\cF$ is invariant under the action of $\mathbf{B}$ if and only if it is a nearly free vector bundle with jumping point $p$.
\end{itemize}
\end{THM}

Finally, in Section \ref{sec-specialconf}, we investigate a little more deeply the relation, for a rank $2$
vector bundle $\mathcal{F}$ on $\p^2$, between the invariance for the action of a given group and the geometrical configuration of its jumping locus. In this direction, recall that if we consider the whole PGL(3), 
hence $\cF$ to be homogeneous, then all the lines $L \subset \p^2$ induce the same splitting type. Vector bundles 
satisfying such property are called \textit{uniform}. \\
It is known, due to the work of many (see for example \cite{Ballico, Elenc1, EHS, ElliaUnif, Sato, VandeVen}), that every uniform vector bundle on $\p^n$ with rank $r \leq n+1$ is also homogeneous. On the other hand, it has been of interest to find examples of uniform vector bundles, which are not homogeneous, of the lowest possible rank, see for example \cite{Elenc-UnifNoHom, Drezet, MMR} and \cite[Thm 3.3.2]{OSS}.\\
We will observe that the  equivalence between the invariance and the jumping locus is already broken for the rank 2 case, 
and we will provide two examples of infinite families of vector bundles which are \textit{almost uniform},
i.e. whose jumping locus is given by all the lines passing through a fixed point $p$, all having the same order of jump. At the same time, the obtained examples are not \textit{almost homogeneous}, i.e. they are not invariant under the action of the considered groups. Notice that our definition of almost uniform vector bundles may differ from the one considered by other authors, where it means to have a finite number of jumping lines (see for example \cite{ElliaJump}).\\

\noindent\textbf{Acknowledgements.} We wish to thank the referee for suggesting major improvements to
the paper and pointing out the references \cite{KaneyamaP2,KaneyamaPn}.

\section{Action of $G_p$, $G_L$ and $\bf{B}$}
 We consider now the subgroup 
$G_p=\mathrm{Stab}_p(\mathrm{PGL}(3,\C))$  that fixes a point $p\in \p^2$,  the subgroup $G_L=\mathrm{Stab}_p(\mathrm{PGL}(3,\C))$  that fixes a line $L \subset \p^2$ and, 
when $p\in L$, we consider also the subgroup defined by the intersection $\mathbf{B}=G_p\cap G_L$. Denote by $p^\vee$ and $L^\vee$, respectively, the associated line and point in the dual projective plane. In order to have a good description of the matrices representing the elements of the considered groups, let us choose the point $p=(1:0:0)$ and the line $L=\{z=0\}$ in $\p^2$.

\smallskip

In this section we describe the action of these three subgroups of $\mathrm{PGL}(3)$.\\
First of all, notice that they can be described as subgroups of matrices in the following way:
$$G_p=\{ \left[
\begin{array}{ccc}
1 & * & * \\
0 & a & b \\
0 & c & d
\end{array}
\right],\,\, ad-bc\neq 0 \},\,\,
G_L=\{ \left[
\begin{array}{ccc}
a & b & * \\
c & d & * \\
0 & 0 & 1
\end{array}
\right],\,\,  ad-bc\neq 0 \} \,\, \mathrm{ and } $$
$$\mathbf{B}=\{ \left[
\begin{array}{ccc}
a & * & * \\
0 & b & * \\
0 & 0 & c
\end{array}
\right],\,\, abc\neq 0 \}.$$ 

\subsection{Action of $G_p$}\label{sec-GP}
First of all, let us describe how the group $G_p$ acts on the points and lines of the projective plane.
\begin{lem}
\label{Lem1}
 The group $G_p$ acts transitively on the following sets: 
 \begin{enumerate}
  \item  points of $\p^2\setminus \{p\}$,
  \item  lines $L$ such that  $p\in L$,
  \item  lines $L$ such that $p\notin L$.
 \end{enumerate}
\end{lem}
\begin{proof}
 It is clear that the action of $G_p$ on these three sets is well defined. We would like to prove that these actions are transitive. In order to prove item (1), we recall that $G$ acts transitively on 
 the set of quadruples of points of $\p^2$, hence the subgroup $G_p$ acts transitively on the set of triples of $\p^2\setminus \{p\}$. This $4$-transitivity of $G$ implies that  $G$ acts transitively on the pair of lines, which moreover implies the transitivity of the action of $G_p$ on both sets of lines, proving the last two items.
\end{proof}

\subsection{Action of $G_L$}\label{sec-GL}
Let us focus now on the subgroup $G_L$.

\begin{lem}\label{lem-tranGL}
 The group $G_L$ acts transitively on the following sets: 
 \begin{enumerate}
  \item  points of $\p^2\setminus L$,
  \item  points of the line $L$,
  \item pairs of points $\{(p_1,p_2)\in (\p^2\setminus L)^2\,|\, p_1\neq p_2\}.$
 \end{enumerate}
\end{lem}
\begin{proof}
All items can be proven directly choosing appropriate matrices.\\
To prove the first one, notice that the matrix 
  $$ \left[
\begin{array}{ccc}
1 & * & a \\
0 & 1 & b \\
0 & 0 & c
\end{array}
\right]$$
sends the point $(0:0:1)$ to any point $(a:b:c)$ with $c\neq 0$.

To prove the second one, notice that the matrix 
$$ \left[
\begin{array}{ccc}
a & b & * \\
c & 1 & * \\
0 & 0 & 1
\end{array}
\right]$$
sends the point $(1:0:0)$ to any point $(a:c:0)$. Observe that if $a=0$, we ask that $b \neq 0$.

The matrix  $$\left[
\begin{array}{ccc}
a &c &  0 \\
b & d& 0 \\
0 &0 & 1
\end{array}
\right]$$ sends  the pair of points $\{(1:0:1),(0:1:1)\}$ to any other pair 
$\{(a:b:1),(c:d:1)\}$.
\end{proof}

\subsection{Action of $\mathbf{B}$}\label{sec-T}
As done for the previous groups, let us prove the transitivity properties of $\mathbf{B}$ that will be needed.

\begin{lem}\label{invariant-t}
 The group $\mathbf{B}$ acts transitively on the following sets: 
 \begin{enumerate}
  \item  $L\setminus \{p\}$ in the projective plane,
  \item  $\check{\p^2}\setminus \{p^{\vee}\}$ in the dual projective plane.
  \item  $ \{p^{\vee}\}\setminus \{L^{\vee}\}$ in the dual projective plane.
 \end{enumerate}
\end{lem}
\begin{proof}
To prove the first item, it is sufficient to observe that any point $(u:1:0)\in L\setminus \{p\}$ is the image of $(0:1:0)$ by $$ \left[
\begin{array}{ccc}
1 &u &  0\\
0 & 1& 0\\
0 &0 & 1
\end{array}
\right].$$
To prove the second item, it is enough to show that the line $x=0$ can be sent to any line $L_{v,w}=\{x=vy+wz\}$ by an element of the group $\mathbf{B}$.
Indeed the matrices ($c\in \C$)
$$ \left[
\begin{array}{ccc}
1 &v & cv+w\\
0 & 1& c\\
0 &0 & 1
\end{array}
\right]$$ send the point $(0:\alpha:\beta)$ to $(\alpha v+\beta (cv+w):\alpha +c\beta: \beta)\in L_{v,w}.$ 

\medskip

To prove the third item, we have to show that any line $y+wz=0$ can be sent to another line of the same type, i.e. $y+w^{'}z=0$. It is enough to show that 
$y=0$ can be sent to the line $y+wz=0$, for any $w\in \C$.\\ 
The matrices of type
$$ \left[
\begin{array}{ccc}
1 &a& b\\
0 & 1& -w\\
0 &0 & 1
\end{array}
\right],$$
 for $(a,b)\in \C$, are the required ones.
\end{proof}

\section{Nearly free vector bundles}
In this section we will focus on the family of nearly free vector bundles, proving that they are invariant for the action of the groups $G_p$ and $\mathbf{B}$, but not for the action of $G_L$.\\

Consider the family of rank two vector bundles $\mathcal{E}_n$ parametrized by the positive integers $n\in \N^*$ and defined by the following short exact sequence:
$$
\begin{CD}
 0@>>>  \sO_{\p^2}(-n) @>(y,z,x^n)>> \sO_{\p^2}^2(1-n) \oplus \sO_{\p^2}@>>> \mathcal{E}_n @>>>  0.
\end{CD}
$$
These bundles belong to the family of \textit{nearly free vector bundles}, first introduced in \cite{DimcaSticlary-NF}. They belong also to the family of indecomposable $\mathbf{T}$ invariant bundles described by  Kaneyama in \cite{KaneyamaP2}, indeed they verify 
$\mathcal{E}_n=E(1,1,n)$. Notice that when $n=1$ we have $\mathcal{E}_1=T_{\p^2}(-1)$, when $n=2$ the bundle $\mathcal{E}_2$ is semi-stable and when $n\ge 3$ the bundle $\mathcal{E}_n$ is unstable, because in this case $c_1(\mathcal{E}_n)<0$ and $\HH^0(\mathcal{E}_n)\neq 0$.\\
In \cite{MV}, we have proven that, fixing a point $p\in \p^2$ and supposing that $p=(1:0:0)$ up to a change of coordinates, 
the point $p$ determines any nearly free vector bundle up to isomorphism and they are all described by the same exact sequence as the one defining $\mathcal{E}_n$. 
The point $p$ has been denominated \textit{jumping point}, indeed a line is a jumping line for $\mathcal{E}_n$ when $n>1$ if and only if it passes through $p$.
Moreover, when $n>1$, such point appears as the zero locus of the unique non zero global section of $\HH^0(\mathcal{E}_n)$.

\medskip

Let us denote by $NF(p)=\{\mathcal{E}_n, n\in \N^*\}$ the set of all such bundles.\\
Since it is clear that a direct sum of two line bundles is invariant under 
the action of any subgroup of $\mathrm{PGL}(3)$, we can consider only the group action on indecomposable bundles.

\medskip

We conclude this section studying the action, and the possible invariance, of the considered groups on the bundles in $NF(p)$.
\begin{lem}\label{NF-invariant} The behaviour of $NF(p)$ under the action of the three considered subgroups is the following:
\begin{enumerate}
 \item Any element in  $NF(p)$ is invariant under the action of $G_p$ and $\mathbf{B}$.
 \item The only invariant vector bundle in $NF(p)$ under the action of $G_L$ is 
 $\mathcal{E}_1=T_{\p^2}(-1)$.
\end{enumerate}
\end{lem}
\begin{proof}
(1) Let us prove first that any element in  $NF(p)$ is $G_p$-invariant. Being $\mathbf{B} = G_p \cap G_L$, this will also prove the invariance of any bundle in $NF(p)$ under the action of $\mathbf{B}$.

 Consider the dual exact sequence 
$$
\begin{CD}
 0@>>> \mathcal{E}^{\vee}_n @>>> \sO_{\p^2}^2(n-1) \oplus \sO_{\p^2} @>(y,z,x^n)>>  \sO_{\p^2}(n)@>>>  0,
\end{CD}
$$
and the action, on the sequence, given by the element $g=\left[
\begin{array}{ccc}
1 & u & v \\
0 & a & b  \\
0 & c & d  \\
\end{array}
\right]\in G_p$.\\
We obtain 
$$
\begin{CD}
 0@>>> g^{*}\mathcal{E}^{\vee}_n @>>>  \sO_{\p^2}^2(n-1) \oplus \sO_{\p^2} @>(ay+bz,cy+dz,(x+uy+vz)^n)>>  \sO_{\p^2}(n)@>>>  0,
\end{CD}
$$
which fits into the following commutative diagram
$$\begin{CD}
0 @>>> g^{*}\mathcal{E}^{\vee}_n   @>>>  \sO_{\p^2}^2(n-1) \oplus \sO_{\p^2} @>(ay+bz,cy+dz,(x+uy+vz)^n)>> \sO_{\p^2}(n) @>>> 0\\
  @.   @V{\simeq}VV    @V{N}VV  @VVV \\
  0 @>>> \mathcal{E}^{\vee}_n  @>>>  \sO_{\p^2}^2(n-1) \oplus \sO_{\p^2} @>(y,z,x^n)>> \sO_{\p^2}(n) @>>> 0,
\end{CD}
$$ with $$N=\left[
\begin{array}{ccc}
a & c & f_{n-1} \\
b & d & g_{n-1} \\
0 & 0 & 1
\end{array}
\right]$$ and $f_{n-1}$ and $g_{n-1}$ are degree $n-1$ polynomials verifying $$(x+uy+vz)^n=x^n+yf_{n-1}+zg_{n-1}.$$
The isomorphism induced in the left map of the diagram proves the invariance.\\

\noindent(2) Let us prove the second item.\\
If $n=1$ the bundle $\mathcal{E}_1$ is homogeneous and fixed by the whole group $\mathrm{PGL}(3)$.\\
If $n>1$, we have recalled that the bundle $\mathcal{E}_n$, and therefore the locus of the jumping lines, is determined by its jumping point $p$. 
But this point is not fixed by all the elements in $G_L$, proving that $\mathcal{E}_n$ is not invariant under its action.  
\end{proof}

\section{Splitting type of invariant bundles}\label{sec-splitting}
The splitting type of a vector bundle is the way its restriction on a given line splits. In this section we will describe the possible splitting types of vector
bundles which are invariant under the considered subgroups.\\

As we will see better in Section \ref{sec-specialconf}, this geometric description of the jumping locus is not equivalent to the invariance.

\medskip

Let us now consider a rank two vector bundle $\mathcal{F}$ which is not uniform.
Then the splitting type is constant for any line $l$ belonging to a non empty open set $ U\subset \PD^2$. This is usually referred to as the \textit{general splitting type}. 
Let us assume that  $\mathcal{F}\otimes \sO_l=\sO_l(a)\oplus \sO_l(b)$ for $l \in U$, denoting by $\delta=|a-b|$ the gap appearing for such general splitting. This gap is minimal on $U$, in other words,
for any line $l$ in the projective plane, we have that $\delta \le \delta(l)=|a_l-b_l|$ where $\mathcal{F}\otimes \sO_l= \sO_l(a_l)\oplus \sO_l(b_l)$.\\
Since $\mathcal{F}$ is not uniform,  the \textbf{set} $S(\mathcal{F})_{\textrm{set}}=\{l, \delta(l)> \delta\}$, whose lines are called \textit{jumping lines}, is not empty.\\
These jumping lines posses a scheme structure that we will denote by $S(\mathcal{F})$. 

To simplify the description of this scheme let us assume that the bundle
$\mathcal{F}$ is normalized, that is $c_1=c_1(\mathcal{F})\in \{-1,0\}$. 

\begin{itemize}
\item When $c_1=0$ and $\mathcal{F}$ is stable (resp. semistable) then $\delta=0$ and  $S(\mathcal{F})$ is a curve (resp. a union of lines that can have multiplicity) of degree $c_2(\mathcal{F})$.

\item  When $c_1=-1$ and $\mathcal{F}$ is stable then $\delta=1$ and $S(\mathcal{F})$ is, in general, a finite scheme but it could also contain a divisor. 
If it is a finite scheme then its length is the binomial number $\binom{c_2(\mathcal{F})}{2}$.

\item  When $\mathcal{F}$ is unstable then $\delta \ge 2$ and there exists $n>0$ such that $\HH^0(\mathcal{F}(c_1-n))\neq 0$ 
and $\HH^0(\mathcal{F}(c_1-n-1))=0$. Then $\hh^0(\mathcal{F}(c_1-n))=1$ and the unique non zero section vanishes in codimension $2$.
The lines meeting this zero scheme form the scheme $S(\mathcal{F})$ of jumping lines. 
\end{itemize}

\subsection{Splitting type of $G_p$-invariant bundles}\label{subsec-splitGp}
\begin{lem}
\label{Lem-technique} Let $\mathcal{F}$ be a non uniform $G_p$-invariant rank two vector bundle on $\p^2$. Then, 
$S_{\mathrm{set}}(\mathcal{F})=p^{\vee}$ and $\delta(l)$ is constant for any $l\ni p.$
\end{lem}
\begin{rmk}
 This is a set-theoretic description. The scheme of jumping lines is then a multiple structure on $p^{\vee}$.
\end{rmk}
\begin{proof}
Since $\mathcal{F}$ is not uniform, $S(\mathcal{F})$ has dimension at most $1$. By Lemma  \ref{Lem1},  $S(\mathcal{F})$ cannot contain a line $L$ such that $p\notin L$ and, 
because of the invariance combined with the transitivity of the chosen action, it coincides with the whole set $p^{\vee}$ of lines through $p$. Moreover, any jumping line has the same splitting type. 
\end{proof}

\subsection{Splitting type of $G_L$-invariant bundles}
\begin{lem}
\label{Lem-technique-GL} Let $\mathcal{F}$ be a non uniform $G_L$-invariant rank two vector bundle on $\p^2$. Then, 
$\mathcal{F}$ is stable, $c_1=-1$ and 
$S(\mathcal{F})$ is a finite scheme of length $\frac{c_2(\mathcal{F})(c_2(\mathcal{F})-1)}{2}$ supported by  $\{L^{\vee}\}$.
\end{lem}
\begin{proof}
Since by Lemma \ref{lem-tranGL} the group $G_L$ acts transitively on $\check{\p^2} \setminus \{L^{\vee}\}$, the set $S_{\mathrm{set}}(\mathcal{F})$ cannot contain
a line distinct from $L$. As we said before, if $c_1(\mathcal{F})=0$ or if $\mathcal{F}$ is unstable then its scheme of jumping lines contains necessarily a curve. 
Then, $\mathcal{F}$ is stable and $c_1=-1$. Since its scheme of jumping lines is finite its length is $\frac{c_2(\mathcal{F})(c_2(\mathcal{F})-1)}{2}$.
\end{proof}

\subsection{Splitting type of $\mathbf{T}$-invariant bundles}
\begin{lem}
Let $\mathcal{F}$ be a non uniform $\mathbf{T}$-invariant rank two vector bundle on $\p^2$ as presented in (\ref{bundle-torusinvariant}), i.e. 
$$
\mathcal{F} \simeq E(a,b,c) \otimes \sO_{\p^2}\left(\frac{a+b+c-c_1}{2}\right).$$
Suppose $a\leq b \leq c$ and denote $p_1 =(1:0:0)$, $p_2 =(0:1:0)$, $p_3 =(0:0:1)$, $L_1:=\{x=0\}$, $L_2:=\{y=0\}$ and $L_3:=\{z=0\}$.\\ 
Then, $S_{\mathrm{set}}(\mathcal{F})\subset \left\{p_1^{\vee},p_2^{\vee},p_3^{\vee}\right\}$ and the order of jump depends on the stability and the mutual relations between $a, b$ and $c$. More precisely, 
the possible splitting types are:
 \begin{itemize}
 \item $\mathcal{F}_l=\sO_l(k)\oplus \sO_l(-k+c_1)$ when $l^{\vee}\in \check{\p^2}\setminus \left\{p_1^{\vee},p_2^{\vee},p_3^{\vee}\right\}$. If $\mathcal{F}$ is stable, then $k=0$. If not, $k= \frac{c-a-b+c_1}{2}$ for $a+b \leq c$ and $k=\frac{a+b-c+c_1}{2}$ for $c \leq a+b$.
 \item $\mathcal{F}_l=\sO_l\left(\frac{c-a-b+c_1}{2}\right)\oplus \sO_l\left(\frac{a+b-c+c_1}{2}\right)$ when $l^{\vee} \in  p_1^{\vee}\setminus \{L_2^{\vee}\}$ or $l^{\vee} \in  p_2^{\vee}\setminus \{L_1^{\vee}\}$.
\item $\mathcal{F}_l=\sO_l\left(\frac{b-a-c+c_1}{2}\right)\oplus \sO_l\left(\frac{a+c-b+c_1)}{2}\right)$ when $l^{\vee} \in  p_3^{\vee}\setminus \{L_1^{\vee}\}$.
\item $\mathcal{F}_l=\sO_l\left(\frac{a-b-c+c_1}{2}\right)\oplus \sO_l\left(\frac{b+c-a+c_1}{2}\right)$ when $l^{\vee} = L_1^{\vee}$.
\end{itemize}
\end{lem}
\begin{proof}
Having an explicit description of the defining matrix of the bundle, given in (\ref{bundle-torusinvariant}) as
$$
A = \left[x^a \:\:\: y^b \:\:\: z^c\right],
$$
we can take its restriction on the considered lines. For example, if $l:=\{y = \alpha x \mid \alpha \neq 0\}$, we get, after linear combinations of its columns, $A_l = \left[x^a \:\:\: 0 \:\:\: z^c\right]$, from which the splitting type follows directly. All the other cases can be done analogously.
\end{proof}

\subsection{Splitting type of $\mathbf{B}$-invariant bundles}\label{sec-splitT}
\begin{lem}
\label{Lem-technique-T} Let $\mathcal{F}$ be a non uniform $\mathbf{B}$-invariant rank two vector bundle on $\p^2$. Then, 
$S_{\mathrm{set}}(\mathcal{F})=p^{\vee}$ or $S_{\mathrm{set}}(\mathcal{F})=\{L^{\vee}\}$. The second case cannot occur if 
$c_1=0$ or if $\mathcal{F}$ is unstable. More precisely, 
there are three possible splitting types:
 \begin{itemize}
 \item $\mathcal{F}_l=\sO_l(k)\oplus \sO_l(-k+c_1)$  with $k\ge 0$ when $l^{\vee}\in \check{\p^2}\setminus p^{\vee}$. 
 \item $\mathcal{F}_l=\sO_l(k+h)\oplus \sO_l(-k-h+c_1)$ with $h\ge 0$  when $l^{\vee} \in  p^{\vee}\setminus \{L^{\vee}\}$.
\item $\mathcal{F}_l=\sO_l(k+h+i)\oplus \sO_l(-k-h-i+c_1)$ with $h\ge 0, i\ge 0$  when $l=L$. 
\end{itemize}
\end{lem}
 
\begin{proof}
 Denote the generic splitting of $\mathcal{F}$ by $\sO_l(k)\oplus \sO_l(-k+c_1)$; in particular, if $\mathcal{F}$ is stable (or semi-stable), we have $k=0$.
 The descriptions given in the second and third item follow directly from the transitivity of the action of $\mathbf{B}$ described in Lemma \ref{invariant-t}.
 Indeed, all the (possible) jumping lines are the ones passing through the point $p$; moreover,
 they must have all the same splitting type except for the line $L$ fixed by the action, where the gap $\delta_L$ could be bigger.\\
 Specifically, if $h>0$ then the set of jumping lines is the line in the dual projective plane  
 $S_{\mathrm{set}}(\cF)=p^{\vee}$; on the other hand, if $h=0$ and $i>1$, this set is just a point $S_{\mathrm{set}}(\cF)=\{L^{\vee}\}$. 
\end{proof}

\begin{rmk}
Under the hypothesis and using the notation of the previous result, notice that the case $h=0$ and $i>0$, which gives that $S_{\mathrm{set}}(\cF)=p^{\vee}$, can only occur if if $c_1=-1$ and $\mathcal{F}$ stable.

\end{rmk}

\section{Main results}\label{sec-main}
In this section we will finally characterize the vector bundles which are invariant under the action of the considered subgroups.
\begin{thm}\label{thm-invariant-t}
 A non decomposable rank two vector bundle $\cF$ is $\mathbf{B}$-invariant if and only if it is a nearly free vector bundle.
\end{thm}
Before proving this theorem, let us show that it implies the characterization of the invariance for the subgroups $G_p$ and $G_L$.
\begin{corol}
  A non decomposable rank two vector bundle $\cF$ is $G_p$-invariant if and only if it belongs to $NF(p)$.
\end{corol}
\begin{proof}
 We have already seen that any bundle in $NF(p)$ is $G_p$-invariant.
 
 Conversely, we have to prove that these are the only invariant bundles. Since $\mathbf{B}=G_p\cap G_L$, the invariant bundles under the action of $G_p$ must also be 
invariant under the action of $\mathbf{B}$, proving the result.
\end{proof}

\begin{corol}
 A rank two vector bundle $\mathcal{F}$ is invariant under the action of $G_L$ if and only if $\mathcal{F}$ is homogeneous.
\end{corol}
\begin{proof}
Theorem \ref{thm-invariant-t} proves that the invariant bundles under the action of $\mathbf{B}=G_p\cap G_L$ all belong to $NF(p)$. Moreover, Lemma \ref{lem-tranGL} shows that a nearly free bundle $\cF$ that is not homogeneous is not $G_L$-invariant, proving the result.
\end{proof}

\begin{proof}[Proof of Theorem \ref{thm-invariant-t}]
 Let $\cF$ be a non decomposable normalized rank two vector bundle which is $\mathbf{B}$-invariant. Since it is necessarily also $\mathbf{T}$-invariant, there exist three positive integers $a,b, c$ such that $\cF(n)\simeq E(a,b,c)$ where $n=\frac{a+b+c-c_1(\cF)}{2}$. We can assume without loss of generality that $a\le b\le c$. Consider the exact sequence
    $$ 
     0\longrightarrow \cF^{\vee}(-n)  \longrightarrow \sO_{\p^2}(-a)\oplus \sO_{\p^2}(-b)\oplus \sO_{\p^2}(-c) \stackrel{(x^a,y^b,z^c)}\longrightarrow  \sO_{\p^2} \longrightarrow 0.
$$
When $a=b=1$, these bundles are nearly free and we have already proved that they are $\bf{B}$-invariant.

\smallskip

Conversely, assume that $b>1$.  Let $g\in \bf{B}$ correspond to the transformation 
$$ (x,y,z) \mapsto (x+y,y,z).$$
Since $g^*\cF^{\vee}=\cF^{\vee}$, this implies (see Lemma \ref{NF-invariant} where this argument is previously used) the existence of a square invertible matrix $N$ fitting in
  $$ \begin{CD}
     \sO_{\p^2}(-a)\oplus \sO_{\p^2}(-b)\oplus \sO_{\p^2}(-c) @>N>>  \sO_{\p^2}(-a)\oplus \sO_{\p^2}(-b)\oplus \sO_{\p^2}(-c) 
    \end{CD}
$$
such that $N^{t}(x^a,y^b,z^c)={}^t((x+y)^a, y^b,z^c).$ 
Let $(\alpha,\beta, \gamma)$ be the first line of $N$. Since $\alpha$ is a non zero constant it can be assumed to be $1$. Then we should have 
$$ x^a+\beta y^b+\gamma z^c=(x+y)^a.$$
This would give $\gamma=0$, $\beta=1$ and $a=b=1$ which contradicts the assumption.

\end{proof}

\begin{rmk}
We would like to observe that, in his work, Kaneyama does not consider the geometric configuration of the locus of jumping lines. It is possible to prove Theorem \ref{thm-invariant-t} only by studying such configuration combined with the invariance, but this leads to a much less direct proof.
\end{rmk}

\section{Special geometric configurations of the jumping locus}\label{sec-specialconf}
In Section \ref{subsec-splitGp} we have noticed that if a rank 2 vector bundle $\mathcal{F}$ on $\p^2$ is invariant under the action of $G_p$, then, if not uniform, 
its jumping locus is given by all lines passing through $p$. Moreover, all jumping lines have the same order.\\
It is compelling to ask ourselves the natural question
\begin{quote}
\textit{Is the invariance equivalent to the obtained special geometric configuration of the jumping locus?}
\end{quote} 
From \cite[Theorem 2.8]{MV}, we already know the answer to be negative. Nevertheless, the result can be generalized to any order of jump, finding interesting families of stable bundles.

From the description recalled at the beginning of Section \ref{sec-splitting}, we have that a non decomposable rank two vector bundle $\cF$ with even $c_1$ has a curve of jumping lines.
Assume that $\cF$ is normalized and that such curve is a line, eventually with multiplicity. The following result implies that if $\cF$ is stable and it has the described geometric configuration of the jumping locus, then its first Chern class is odd.
\begin{thm}\label{thm-semistable}
 A non decomposable rank two vector bundle $\cF$ such that $c_1(\cF)=0$ and $S_{\mathrm{set}}(\cF)=p^{\vee}$, where $p$ is a point in $\p^2$,
 is either unstable or properly semistable.
\end{thm}
\begin{proof}
Assume that $\cF$ is stable.  The splitting type on a general line $l$ through $p$ is 
 $\sO_l(-h)\oplus \sO_l(h)$ with $h>0$. This means that, for a general line $l$ through $p$ we have $\hh^0(\cF_{|l}(-h))=1$.
 Thanks to this fact, we can construct a special non zero section of $\cF(n-h)$ for some $n\ge h$ in the following way. 
Let us consider
the following diagram, constructed blowing up the point $p$ in $\p^2$
$$
\xymatrix{
\tilde{\p}^2 \ar[d]_{\tilde p} \ar[r]^{\tilde q} & {p}^\vee\\
\p^2
}
$$
Because of the possible splitting types, we get that $\tilde{q}_*\tilde{p}^*(\cF(-h))$ is an invertible sheaf 
on $p^\vee$, that is $\sO_{p^\vee}(-n)$ with $n> h$ thanks to the stability of $\cF$. This gives a non zero map
$$ \tilde{q}^*\sO_{p^\vee} \longrightarrow  \tilde{p}^*(\cF(-h))\otimes \tilde{q}^*\sO_{p^\vee}(n).$$
Remind that $\tilde{p}_*\tilde{q}^*\sO_{p^\vee}(n)=\cI_p^n(n)$, then taking the direct image on $\p^2$, we obtain a short exact sequence
 \begin{equation}\label{stable-firststep} 
 \begin{CD}
  0@>>> \sO_{\p^2} @>>> \cF(n-h) @>>> \cI_{\Gamma}(2n-2h+c_1) @>>> 0.
 \end{CD}
\end{equation}
with $\cI_{\Gamma}\subset \cI_p^n$. 
This gives some numerical conditions. The length of $\Gamma$ is $c_2(\cF(n-h))=c_2(\cF)+(n-h)^2$ and this length is greater or equal than $n^2$ by construction.
 Indeed, locally at $p=(1:0:0)$, the zero set $\Gamma$ is a 
complete intersection defined by two polynomials in $\oplus_{k\ge 0}\HH^0(\cI_p^n(n+k))$, in particular its length is at least $n^2$. 
 This means that we have 
$$c_2(\cF)\ge h(2n-h).$$
Moreover, since $\cF$ is stable, $c_2(\cF)$ is the degree of the curve $S(\cF)$ which means that, if $f=0$ is the linear form defining $p^{\vee}$ in $\check{\p^2}$ then 
$S(\cF)$ is defined by $f^{c_2(\cF)}=0$.

Let $l$ be a general jumping line.\\ 
Using the method implemented by Maruyama in \cite{M} to determine the multiplicity of the singular point $l$
of the curve of jumping lines of $\cF$, we consider an elementary transformation of $\cF$, given by the jumping line, which induces the following diagram:
\begin{equation}\label{Maruyama-diag}
\xymatrix{
& & 0 \ar[d] & 0\ar[d]\\
0 \ar[r] & \sO_{\p^2}(h-n) \ar[r] \ar[d]_{\simeq} & \cF_1 \ar[r] \ar[d] & \cI_{\Gamma_1}(n-h-1) \ar[d] \ar[r] & 0 \\
0 \ar[r] & \sO_{\p^2}(h-n) \ar[r] & \cF \ar[r] \ar[d] & \cI_{\Gamma}(n-h) \ar[d] \ar[r] & 0 \\
& & \sO_l(-h) \ar[d] \ar[r]^{\simeq} & \sO_l(-h) \ar[d]\\
& & 0 & 0
}
\end{equation}
Recall that Maruyama's method states that, denoting $(\cF_i)_{|l}\simeq \sO_l(a_i) \oplus \sO_l(b_i)$ with $a_i \geq b_i$, 
the multiplicity is computed as $\mult_{l^\lor}S(\sF) = \sum_{j=0}^k a_i$, where $a_k$ is the first integer in the decreasing sequence 
$\{a_i\}_{i\geq 0}$ with $a_k\leq0$. Iterating the previous diagram, using subsequent elementary transformations, 
we get that $\mult_{l^\lor}S(\sF)\leq hn$. Combining the obtained inequalities, we have that $h(2n-h)\le nh$. The only possibility is $h=n$ which proves that $\cF$ is semistable but not stable.

The main ingredient to prove the last inequality is to look at the local description of the previous diagram. Locally at the point $p=(1:0:0)$, $\Gamma$ is defined by two non homogeneous polynomials $(f,g)$, that we describe in terms of their homogeneous components
$$
f = \sum_{k=n}^{\deg(f)} f_k, \:\:\: g = \sum_{k=m}^{\deg(g)} g_k.
$$ 
Observe that the lowest degree of the homogeneous components must be $n$ for one of the two defining polynomials (which we suppose to be $f$) and greater or equal than $n$ for the other one (in our case $m\geq n$). Else, we would have that $\Gamma$ contains the fat point defined by a power of $I_p$ greater then $n$, which is impossible. We will mainly focus on the $n$-th homogeneous component $f_n$ of $f$, which we describe as
$$
f_n = \sum_{i+j=n} \alpha_{i,j}y^i z^j.
$$
Because of the hypothesis on the splitting type for the generic jumping line, we can consider a generic change of coordinates of $\mathbb{P}^2$, which fixes the point $p$, that allows us to suppose that all the coefficients $\alpha_{i,j}$ are non zero and to consider, as generic jumping line, the one defined by $y=0$.

This implies that, because of Diagram \ref{Maruyama-diag}, the polynomial 
$$
f^{(1)} = \frac{f - \left(\alpha_{0,n}z^n + \sum_{t>n} \beta_ t z^t\right)}{y}, 
$$
with $\beta_t$ non zero for a finite number of values of $t$, belongs to the ideal defining $\Gamma_1$. In particular, its homogeneous part of degree $n-1$ comes from $f_n$ and henceforth all the possible monomials $y^i z^j$, in this case with $i+j=n-1$, appear with non zero coefficient. Therefore, we have that 
$$
(\cF_1)_{|l} \simeq  \sO_l(-h+s_1)\oplus \sO_l(h-1-s_1), \:\: \mbox{ with } s_1 \geq 0. 
$$
Observe that the integer $s_1$ appears because, depending on the homogeneous parts of higher degree in the considered polynomials, we could get a lower order, at the point $p$, in the intersection of the line $l$ with $\Gamma_1$.

Iterating the previous process, we have that at the $k$-th step we must have
$$
(\cF_k)_{|l} \simeq  \sO_l(-h+s_k)\oplus \sO_l(h-k-s_k), \:\: \mbox{ with } s_k \geq 0. 
$$
and the iteration end at most at the $(n-1)$-st step. Indeed, we obtain $f^{(n-1)} = z + \cdots$, which gives a splitting for $\cF_{n-1}$ with both degree less or equal than zero.

\end{proof}

We conclude this section providing two infinite and explicit families of vector bundles which are almost uniform, i.e. their jumping loci are given by all the lines passing through a fixed point, but are not almost homogeneous, due to Theorem \ref{thm-invariant-t}. Notice that all the stable bundles described in the second family have odd first Chern class, as implied by the previous result.

\begin{eg}\label{characterization-by-jl-1}[Properly semistable and unstable bundles]
Let $p$ be the point $p=(0:0:1)$, $k\in \N$ be a positive integer, $c_1=\{-1,0\}$ and $f(x,y,z)$ a  $(2r+2k-c_1)$-form such that $f(p)\neq 0$. Then the sheaf $\cF$ defined by 
$$
\begin{CD}
 0@>>>  \sO_{\p^2}(-2r-k+c_1) @>(f,x^r,y^r)>> \sO_{\p^2}(k) \oplus \sO_{\p^2}(-r-k+c_1)^2@>>> \cF @>>> 0
\end{CD}
$$
is a vector bundle. Its first Chern class is $c_1(\cF)=c_1$, it is properly semistable if $c_1=0$ and $k=0$ but unstable if not (that is if $c_1=-1$ or $k>0$).
According to its definition there is an exact sequence 
$$
0 \longrightarrow \sO_{\p^2}(k) \longrightarrow \cF \longrightarrow \mathcal{I}_\Gamma(-k+c_1)\longrightarrow 0,
$$
where $\Gamma$ is the complete intersection $(x^r,y^r)$. Observe that this vector bundle is not invariant under the action of $G_p$ when $r>1$.

\smallskip

Any line $L$ through $p$ intersects $\Gamma$ along a subscheme of length $r$ which implies 
$$ \cF\otimes \sO_L=\sO_{L}(k+r) \oplus \sO_L(-k-r+c_1).$$
On the contrary the splitting type of $\cF$ on a line $l$ that does not meet $\Gamma$ is 
$$ \cF\otimes \sO_L=\sO_{L}(k) \oplus \sO_L(-k+c_1).$$
This shows that $S(\cF)=\{L, L\ni p\}$ and each jumping line is of order $r$.
\end{eg}

\begin{eg}\label{characterization-by-jl-2}[Stable bundles]
Let $p$ be the point $p=(0:0:1)$,  $f(x,y,z)$ be a $(2r+1)$-form such that $f(p)\neq 0$. Then the sheaf $\cF$ defined by 
$$
\begin{CD}
 0@>>>  \sO_{\p^2}(-2r-2) @>(f,x^{r+1},y^{r+1})>> \sO_{\p^2}(-1) \oplus \sO_{\p^2}(-r-1)^2@>>> \cF @>>> 0
\end{CD}
$$
is a stable vector bundle with Chern classes  $c_1(\cF)=-1$ and $c_2(\cF)= (r+1)^2$.
According to its definition there is an exact sequence 
$$
0 \longrightarrow \sO_{\p^2}(-1) \longrightarrow \cF \longrightarrow \mathcal{I}_\Gamma \longrightarrow 0,
$$
where $\Gamma$ is the complete intersection $(x^{r+1},y^{r+1})$.  Observe that  this vector bundle is not invariant under the action of $G_p$ when $r\ge 1$.

\smallskip

Any line $L$ through $p$ intersects $\Gamma$ along a subscheme of lentgh $r+1$ which implies 
$$ \cF\otimes \sO_L=\sO_{L}(r) \oplus \sO_L(-r-1).$$
On the contrary the splitting type of $\cF$ on a line $l$ that does not meet $\Gamma$ is 
$$ \cF\otimes \sO_L=\sO_{L}\oplus \sO_L(-1).$$
This shows that $S(\cF)=\{L, L\ni p\}$ and each jumping line is of order $r$.
\end{eg}

\bibliographystyle{siam}
\bibliography{MV-Invariant.bib}

\end{document}